\documentclass[11pt]{article}

\usepackage[centertags]{amsmath}
\usepackage{amsfonts}
\usepackage{amssymb}
\usepackage{amsthm}
\usepackage{newlfont}
\usepackage{verbatim}

\newtheorem{thm}{Theorem}[section]
\newtheorem*{thm*}{Theorem}
\newtheorem{cor}[thm]{Corollary}

\newtheorem{lem}[thm]{Lemma}

\newtheorem{prop}[thm]{Proposition}
\newtheorem*{prop*}{Proposition}

\newtheorem*{conj*}{Conjecture}

\newtheorem*{dfn*}{Definition}
\theoremstyle{definition}
\newtheorem{rem}[thm]{\textbf{Remark}}
\newtheorem*{rmk*}{Remark}
\newtheorem*{fact*}{Fact}

\theoremstyle{proof}

\newcommand{\norm}[1]{\left\Vert#1\right\Vert}

\newcommand{\abs}[1]{\left\vert#1\right\vert}
\newcommand{\set}[1]{\left\{#1\right\}}
\newcommand{\brac}[1]{\left(#1\right)}
\newcommand{\scalar}[1]{\left \langle #1 \right \rangle}

\newcommand{\Real}{\mathbb{R}}
\def \RR {\mathbb R}
\def \PP {\mathbb P}
\def \EE {\mathbb E}

\newcommand{\vr}{\textrm{V.Rad.}}
\newcommand{\detcov}{\textrm{det Cov}}
\newcommand{\cov}{\textrm{Cov}}
\newcommand{\diam}{\textrm{diam}}
\newcommand{\Hess}{\textrm{Hess}}

\newlength{\defbaselineskip}
\newcommand{\setlinespacing}[1]%
           {\setlength{\baselineskip}{#1 \defbaselineskip}}

\newcommand{\trcov}{\textrm{trace Cov}}
\newcommand{\vol}{\textrm{Vol}_n}
\numberwithin{equation}{section}

\begin{document}

\title{Centroid Bodies and the Logarithmic Laplace Transform - \\ A Unified Approach}
\author{Bo'az Klartag\textsuperscript{1} and Emanuel Milman\textsuperscript{2}}

\footnotetext[1]{School of Mathematical Sciences, Tel-Aviv
University, Tel Aviv 69978, Israel. Supported in part by the Israel
Science Foundation and by a Marie Curie Reintegration Grant from the
Commission of the European Communities. Email: klartagb@tau.ac.il}

\footnotetext[2]{Department of Mathematics,
Technion - Israel Institute of Technology, Haifa 32000, Israel. Supported by ISF, GIF and the Taub Foundation (Landau Fellow).
Email: emilman@tx.technion.ac.il.}

\date{}

\maketitle \abstract{We unify and slightly improve several bounds on
the isotropic constant of high-dimensional convex bodies; in
particular, a linear dependence on the body's $\psi_2$ constant is
obtained. Along the way, we present some new bounds on the volume of
$L_p$-centroid bodies and yet another equivalent formulation of Bourgain's
hyperplane conjecture. Our method is a combination of the
$L_p$-centroid body technique of Paouris and the logarithmic Laplace
transform technique of the first named author. }

\section{Introduction}

This work combines two recent techniques in the study of volumes of
high-dimensional convex bodies. The first technique is due to
G. Paouris \cite{Paouris-IsotropicTail}, and it relies on properties of the $L_p$-centroid
bodies. The second technique was developed by the first named author
\cite{quarter}, and it uses the logarithmic Laplace transform.

\medskip

Suppose that $\mu$ is a Borel probability measure on $\RR^n$ endowed with a Euclidean structure $\abs{\cdot} = \sqrt{\scalar{\cdot,\cdot}}$. We say that $\mu$ is a $\psi_\alpha$-measure ($\alpha > 0$) with constant
$b_{\alpha}$ if:
\begin{equation} \label{eq:psi-alpha}
 \brac{\int_{\Real^n} \abs{\scalar{x,\theta}}^p d\mu(x)}^{\frac{1}{p}} \leq b_\alpha p^{\frac{1}{\alpha}} \brac{\int_{\Real^n} \abs{\scalar{x,\theta}}^2 d\mu(x)}^{\frac{1}{2}} \;\;\; \forall p \geq 2 \;\;\; \forall \theta \in \Real^n ~.
\end{equation}
It is well-known that the uniform probability measure $\mu_K$ on any
convex body $K \subset \RR^n$ is a $\psi_1$-measure with constant
$C$, where $C
> 0$ is a universal constant (this follows from Berwald's inequality \cite{ber}, see also \cite{MilmanPajor}).
Here, as usual, a convex body in $\RR^n$
means a compact, convex set with a non-empty interior. The isotropic
constant $L_K$ of a convex body $K \subset \RR^n$ is the following affine invariant parameter:
$$
L_K := \vol(K)^{-\frac{1}{n}} \left( \detcov(\mu_K) \right)^{\frac{1}{2n}} ~,
$$
where $\vol$ denotes Lebesgue measure and $\cov(\mu_k)$ denotes the covariance matrix of $\mu_K$. The next theorem unifies and slightly improves
several known bounds on the isotropic constant.

\begin{thm} \label{main_thm}
Let $K \subset \RR^n$ denote a convex body whose barycenter lies at the origin, and
suppose that $\mu_K$ is a $\psi_{\alpha}$-measure ($1 \leq \alpha \leq 2$)
with constant $b_{\alpha}$. Then:
\[
L_K \leq C \sqrt{b_\alpha^\alpha n^{1-\alpha/2}} ~,
\]
where $C > 0$ is a universal constant.
\end{thm}

A central question raised by Bourgain \cite{bou_amer, bou_cong} is
whether $L_K \leq C$ for some universal constant $C>0$, for any
convex body $K \subset \RR^n$ (it is well-known that $L_K \geq c$
for a universal constant $c > 0$). This question is usually referred
to as the {\it slicing problem} or {\it hyperplane conjecture}, see
Milman and Pajor \cite{MilmanPajor} for many of its equivalent
formulations and for further background. Plugging $\alpha = 1$ in
Theorem \ref{main_thm}, we match the best known bound on the
isotropic constant, which is $L_K \leq C n^{1/4}$ for any convex
body $K \subset \RR^n$ (see Bourgain \cite{bou_L} and Klartag
\cite{quarter}). In the case $\alpha = 2$, Theorem \ref{main_thm}
yields $L_K \leq C b_2$. This slightly improves upon the previously
known bound, which is:
\begin{equation}
 L_K \leq C b_2 \sqrt{\log b_2} ~, \label{eq_2307}
\end{equation}
due to  Dafnis and Paouris \cite{DP} in the precise form
(\ref{eq_2307}) and to Bourgain \cite{bou_psi2} (with a different
power of the logarithmic factor). Here, as elsewhere in this text,
we use the letters $c, \tilde{c}, C, \tilde{C}, \bar{C}$ etc. to
denote positive universal constants,  whose value may not
necessarily be the same in different occurrences.

\medskip

We proceed by recalling the definition of the $L_p$-centroid bodies
$Z_p(\mu)$, originally introduced by E. Lutwak and G. Zhang in
\cite{LutwakZhang-IntroduceLqCentroidBodies} (under different
normalization), which lie at the heart of Paouris' remarkable work
\cite{Paouris-IsotropicTail}. Given a Borel probability measure
$\mu$ on $\RR^n$ and $p \geq 1$, denote:
$$
h_{Z_p(\mu)}(\theta) = \left( \int_{\RR^n} \abs{\scalar{x,\theta}}^p d\mu(x) \right)^{\frac{1}{p}} \quad , \quad \theta \in \RR^n ~ .
$$
The function $h_{Z_p(\mu)}$ is a norm on $\RR^n$, and it is the
supporting functional of a convex body $Z_p(\mu) \subseteq \RR^n$
(see e.g. Schneider \cite{Schneider-Book} for information on
supporting functionals). Clearly $Z_p(\mu) \subseteq Z_q(\mu) $ for
$p \leq q$.

Now suppose that $K \subset \RR^n$ is a convex body whose barycenter
lies at the origin, and denote $Z_p(K) = Z_p(\mu_K)$, where $\mu_K$
is as before the uniform probability measure on $K$. As realized by
Paouris, obtaining volumetric and other information on $Z_p(K)$ is
very useful for understanding the volumetric properties of $K$
itself. For instance, note that:
\begin{equation} \label{eq:vr-z2}
\vr(Z_2(K)) = \detcov(\mu_K)^{\frac{1}{2n}} ~,
\end{equation}
where the volume-radius of a compact set $T \subset \RR^n$ is
defined as:
$$
\vr(T) = \left( \frac{\vol(T)}{\vol(B_n)} \right)^{\frac{1}{n}} ~,
$$
measuring the radius of the Euclidean ball whose volume equals the
volume of $T$. Here, $B_n = \{ x \in \RR^n ; |x| \leq 1 \}$; 
note that $c n^{-\frac{1}{2}} \leq \vol(B_n)^{\frac{1}{n}} \leq C
n^{-\frac{1}{2}}$, as verified by direct calculation. Furthermore,
it is known (e.g. \cite[Lemma 3.6]{PaourisSmallBall}) that:
\begin{equation} \label{eq:ZnK}
c \cdot Z_\infty(K) \subseteq Z_n(K) \subseteq Z_\infty(K) := conv(K, -K) ~,
\end{equation}
where $conv(K, -K)$ denotes the convex hull of $K$ and $-K$.

A sharp \emph{lower} bound on the volume of $Z_p(K)$ due to Lutwak,
Yang and Zhang \cite{LYZ} states that ellipsoids minimize
$\vr(Z_p(K)) / \vr(K)$ among all convex bodies $K$, for all $p \geq
1$. An elementary calculation yields:
\begin{equation} \label{eq:LYZ-bound}
 \vr( Z_p(K) ) \geq c \sqrt{\frac{p}{n}} \vr(K) \quad \quad
\quad \text{for} \ 1 \leq p \leq n ~,
\end{equation}
which is best possible (up to the value of the constant $c>0$) in terms of $\vol(K)$.
However, in view of the slicing problem and (\ref{eq:vr-z2}), one may try to strengthen (\ref{eq:LYZ-bound}) by replacing its right-hand side
by $c \sqrt{p} \vr(Z_2(K))$.
The next two theorems are a step in this direction.
Before formulating the results, we first broaden our scope.

\medskip

It was realized by K. Ball \cite{BallPhD,Ball-kdim-sections} that
many questions regarding the volume of convex bodies are better
formulated in the broader class of logarithmically-concave measures.
A function $\rho: \RR^n \rightarrow [0, \infty)$ is called
log-concave if $ -\log \rho: \RR^n \rightarrow (-\infty, \infty] $
is a convex function. A probability measure on $\RR^n$ is
log-concave if its density is log-concave. For example, the uniform
probability measure on a convex body and its marginals are all
log-concave measures (see Borell \cite{Borell-logconcave} for a
characterization).


\begin{thm}  \label{main_thm2}
Let $\mu$ be a log-concave probability measure on $\Real^n$ with barycenter at the origin.
Let $1 \leq \alpha \leq 2$, and assume that $\mu$ is a $\psi_\alpha$-measure with constant $b_\alpha$. Then:
\[
\vr(Z_p(\mu)) \geq c \sqrt{p} \vr(Z_2(\mu))~,
\]
for all $2 \leq p \leq C n^{\alpha/2} / b_\alpha^\alpha$. Here $c,C >
0$ denote universal constants.
\end{thm}

Theorem \ref{main_thm} follows immediately from Theorem
\ref{main_thm2}. Indeed, simply observe that for $p$ in the specified range:
\[
c \sqrt{p} \leq \frac{\vr(Z_p(K))}{\vr(Z_2(K))} \leq
\frac{\vr{(conv(K,-K))}}{\vr(Z_2(K)} \leq C \sqrt{n}
\frac{\vol(K)^{1/n}}{\vr(Z_2(K))} =  \frac{C \sqrt{n}}{L_K} ~,
\]
where the last inequality follows from the Rogers-Shephard
inequality \cite{RS}. This completes the proof of Theorem
\ref{main_thm}, reducing it to that of Theorem \ref{main_thm2}. We
remark here that the proof (of both theorems) only requires that the
$\psi_\alpha$ condition (\ref{eq:psi-alpha}) hold for $p \geq 2$
so that $\diam(Z_p(\mu)) \leq c \sqrt{n}$, and only in an average sense (see Subsection \ref{subsec:generalizations}).

\medskip

Our next theorem contains an additional lower bound on the volume of
$Z_p(\mu)$ which complements that of Theorem \ref{main_thm2} in some sense. A
Borel probability measure $\mu$ on $(\RR^n,\abs{\cdot})$ is called isotropic when
its barycenter lies at the origin, and its covariance matrix equals
the identity matrix (i.e. $Z_2(\mu) = B_n$). Any measure with finite
second moments and full-dimensional support may be brought into
isotropic ``position" by means of an affine transformation.

\begin{thm} \label{main_thm3}
Let $\mu$ be an \emph{isotropic} log-concave probability measure on
$\Real^n$. Then:
\[
\vr(Z_p(\mu)) \geq c \sqrt{p} ~,
\]
for all $p \geq 2$ for which:
\begin{equation} \label{eq_0034}
\diam(Z_p(\mu)) \sqrt{\log p} \leq C \sqrt{n} ~.
\end{equation}
Here, $\diam(T) = \sup_{x, y \in T} |x-y|$ stands for the diameter
of $T \subset \RR^n$, and $c,C > 0$ are universal constants.
\end{thm}

Note that the $\psi_\alpha$-condition (\ref{eq:psi-alpha}) is
precisely the requirement that $Z_p(\mu) \subseteq b_\alpha
p^{\frac{1}{\alpha}} Z_2(\mu)$ for all $p \geq 2$, and so the
conclusion of Theorem \ref{main_thm3} agrees with that of Theorem
\ref{main_thm2}, up to the logarithmic factor in (\ref{eq_0034}).
This discrepancy is explained by the fact that in Theorem
\ref{main_thm2}, we actually make full use of the growth of
$\diam(Z_p(\mu))$ for all $p \geq 2$, whereas in Theorem
\ref{main_thm3} we only assumed this control for the end value of
$p$. We emphasize that this constitutes a genuine difference in
assumptions, and that the logarithmic factor in (\ref{eq_0034}) is
not just a mere technicality: we show in Section \ref{sec_counter}
that removing this factor is actually \emph{equivalent} to
Bourgain's original hyperplane conjecture.

We find condition (\ref{eq_0034}) quite interesting from other
respects as well. It is very much related to Paouris' parameter
$q^*(\mu)$, to be discussed in Section \ref{sec4}. In fact, we show
there that the parameter:
\[
q^{\#}(\mu) := \sup \set{q \geq 1 ; \diam(Z_q(\mu)) \leq c^\sharp \sqrt{n} \detcov(\mu)^\frac{1}{2n}} ~,
\]
for a small-enough universal constant $c^\sharp > 0$, is essentially
equivalent to and has the same functionality as Paouris' $q^*(\mu)$
parameter, in addition to being rather convenient to work with.

\medskip The lower bounds in Theorem \ref{main_thm2} and Theorem
\ref{main_thm3} compare with the matching \emph{upper} bounds on
$\vr(Z_p(\mu))$, obtained by Paouris \cite[Theorem 6.2]{Paouris-IsotropicTail}, which are valid for
\emph{all} $2 \leq p \leq n$:
\begin{equation} \label{eq:vr-upper}
\vr(Z_p(\mu)) \leq C \sqrt{p} \vr(Z_2(\mu)) ~.
\end{equation}
This implies that the lower bounds in both theorems above are sharp, up to
constants, and so the only pertinent question is the optimality of the range of $p$'s for which their conclusion is valid.
In this direction, Paouris obtained a partial converse to (\ref{eq:vr-upper}) in the following range of $p$'s:
\begin{equation}
W(Z_p(\mu)) \geq c \sqrt{p} \vr(Z_2(\mu)) \quad \quad \forall 2 \leq p \leq q^\#(\mu)
~. \label{eq_152}
\end{equation}
Here $W(K) = \int_{S^{n-1}} h_K(\theta) d\sigma(\theta)$ denotes half the mean width of $K$,
$\sigma$ is the Haar probability measure on the Euclidean unit sphere
$S^{n-1}$, and $h_K(\theta) = \sup_{x \in K} \scalar{x,\theta}$ is the
supporting functional of $K$. Note that according to the Urysohn inequality, $W(K) \geq \vr(K)$ (see
e.g. \cite{Milman-Schechtman-Book}), and so Theorem \ref{main_thm3} should be thought of as a formal strengthening of (\ref{eq_152}), if it
were not for the logarithmic factor in (\ref{eq_0034}).

\medskip

The rest of this work is organized as follows. We begin with some
more or less known preliminaries in Section \ref{sec2}. In Section
\ref{sec3}, we deduce a new formula for $\vr(Z_p(\mu))$ involving
the ``tilts" of the measure $\mu$ from \cite{quarter, K_psi2}, and
we relate between the $Z_p$-bodies of the original measure and its
tilts. In Section \ref{sec4}, we deviate from our discussion to
review Paouris' $q^*$-parameter, and compare it with  $q^\sharp$;
this section may be read independently of this work. In Section
\ref{sec5}, we use projections and the $q^\sharp$-parameter to
relate between the determinant of the covariance matrix of $\mu$ and
its tilts, and conclude the proofs of Theorems \ref{main_thm2} (in
fact, a more general version) and \ref{main_thm3}. In Section
\ref{sec6}, we show that removing the log-factor from Theorem
\ref{main_thm3} is equivalent to the slicing problem.

\medskip
\noindent \textbf{Acknowledgements}. We thank Grigoris Paouris for interesting discussions.


\section{Preliminaries} \label{sec2}

Given $1 \leq k \leq n$, the Grassmann manifold of all
$k$-dimensional linear subspaces of $\Real^n$ is denoted by
$G_{n,k}$. Given $E \in G_{n,k}$, the orthogonal projection onto $E$
is denoted by $Proj_E$, and given a Borel probability measure $\mu$
on $\Real^n$, we denote by $\pi_E \mu := (Proj_E)_*(\mu)$ the
push-forward of $\mu$ via $Proj_E$. For a convex body $K \subset
\RR^n$ containing the origin in its interior, its polar body is
denoted by:
$$
K^{\circ} = \set{ x \in \RR^n \; ; \; \scalar{x,y} \leq
1  \;\;\; \forall y \in K } ~.
$$
Finally, we denote by $\nabla$ and $\Hess$ the gradient and Hessian
of a sufficiently differentiable function, respectively.

Throughout this text, $x \simeq y$ is an abbreviation for $c x \leq
y \leq C x$ for universal constants $c, C > 0$. Similarly, we write
$x \lesssim y$ ($x \gtrsim y$) when $x \leq C y$ ($x \geq c y$). Additionally, for two convex sets
$K, T \subset \RR^n$ we write $K \simeq T$ when:
$$ c K \subseteq T \subseteq C K $$
for universal constants $c, C > 0$.


\subsection{Extension of the Slicing Problem to log-concave measures} \label{subsec:slicing}

We first recall the well-known extension of the slicing problem from the
class of convex bodies to the class of all log-concave measures, due
to Ball \cite{BallPhD,Ball-kdim-sections}. Given a log-concave
probability measure $\mu$ on $\Real^n$, define its isotropic
constant $L_\mu$ by:
\begin{equation} \label{eq:Lmu}
L_\mu := \norm{\mu}_{L_\infty}^\frac{1}{n} \detcov(\mu)^{\frac{1}{2n}} ~,
\end{equation}
where $\norm{\mu}_{L_\infty} := \sup_{x \in \Real^n} \rho(x)$ and
$\rho$ is the log-concave density of $\mu$. It was shown by Ball
\cite{BallPhD,Ball-kdim-sections} that given $n \geq 1$:
\[
\sup_\mu {L_\mu} \leq C \sup_{K} L_K ~,
\]
where the suprema are taken over all log-concave probability
measures $\mu$ and convex bodies $K$ in $\Real^n$, respectively (see
e.g. \cite{quarter} for the non-even case). Similarly, the following theorem slightly
generalizes Theorem \ref{main_thm}:
\begin{thm} \label{main_thm1+}
Let $\mu$ denote a log-concave probability measure on $\Real^n$ with
barycenter at the origin. Suppose that $\mu$ is in addition a
$\psi_{\alpha}$-measure ($1 \leq \alpha \leq 2$) with constant
$b_{\alpha}$. Then:
\[
L_\mu \leq C \sqrt{b_\alpha^\alpha n^{1-\alpha/2}} ~.
\]
\end{thm}

As was the case with Theorem \ref{main_thm}, deducing Theorem
\ref{main_thm1+} from Theorem \ref{main_thm2} is equally elementary.
We only require the following additional well-known lemma, which
will come in handy in other instances in this work as well. This
lemma serves as an extension of (\ref{eq:ZnK}) to the class of
log-concave measures.

\begin{lem} \label{lem:Zn}
Let $\mu$ denote a log-concave probability measure on $\Real^n$ with barycenter at the origin. Then:
\[
\vr(Z_n(\mu)) \simeq 
\frac{\sqrt{n}}{\norm{\mu}_{L_\infty}^\frac{1}{n}} ~.
\]
\end{lem}

Given Lemma \ref{lem:Zn}, the reduction of Theorem \ref{main_thm1+}
to Theorem \ref{main_thm2} is indeed immediate, since for $p \leq n$
in the range specified in the latter:
\[
c \sqrt{p} \leq \frac{\vr(Z_p(\mu))}{\vr(Z_2(\mu))} \leq \frac{\vr(Z_n(\mu))}{\detcov(\mu)^\frac{1}{2n}} \simeq
\frac{\sqrt{n}}{\norm{\mu}_{L_\infty}^\frac{1}{n} \detcov(\mu)^\frac{1}{2n}} = \frac{\sqrt{n}}{L_\mu} ~.
\]

\begin{proof}[Proof of Lemma \ref{lem:Zn}]
Denote by $\rho$ the log-concave density of $\mu$.
According to \cite[Proposition 3.7]{PaourisSmallBall} (compare with \cite[Lemma 2.8]{K_psi2} and Lemma \ref{lem:Lambda-Zp} below):
\[
 \vr(Z_n(\mu)) \simeq \frac{\sqrt{n}}{\rho(0)^\frac{1}{n}} ~.
\]
However, according to Fradelizi \cite{fradelizi}:
\[
e^{-n} M \leq \rho(0) \leq M \quad , \quad M := \norm{\mu}_{L_\infty} = \sup_{x \in \Real^n} \rho(x) ~,
\]
and so the assertion immediately follows.
\end{proof}


\subsection{$\Lambda_p$-bodies}

Now suppose that $\mu$ is an arbitrary Borel probability measure on
$\RR^n$. Its logarithmic Laplace transform is defined as:
$$
\Lambda_{\mu}(\xi) := \log \int_{\RR^n} \exp(\scalar{\xi, x})
d\mu(x) \quad \quad , \quad \xi \in \RR^n ~.
$$
The function $\Lambda_{\mu}$ is always convex (e.g. by H\"{o}lder's
inequality), and clearly $\Lambda_{\mu}(0) = 0$. If in addition the
barycenter of $\mu$ lies at the origin, then $\Lambda_\mu$ is
non-negative (by Jensen's inequality). In this case, for any $t \geq
0$ and $\alpha \geq 1$:
\begin{equation} \label{eq:Lambda-inclusion}
\frac{1}{\alpha} \set{ \Lambda_\mu \leq \alpha t } \subseteq \set{
\Lambda_\mu \leq t } \subseteq \set{ \Lambda_\mu \leq \alpha t } ~,
\end{equation}
where we abbreviate $\set{\Lambda_{\mu} \leq t} = \{ \xi \in \RR^n ; \Lambda_{\mu}(\xi) \leq t \}$. 
When $\mu$ is log-concave, the convex function $\Lambda_{\mu}$
possesses several additional regularity properties. For instance
$\set{\Lambda_{\mu} < \infty}$ is an open set, and $\Lambda_{\mu}$
is $C^{\infty}$-smooth and strictly-convex in this open set (see,
e.g., \cite[Section 2]{K_psi2}).

The following lemma describes a certain equivalence, known to
specialists, between the $L_p$-centroid bodies and the level-sets of
the logarithmic Laplace Transform $\Lambda_\mu$. See Lata\l a and
Wojtaszczyk \cite[Section 3]{LW} for a proof of a dual version in
the symmetric case (i.e., when $\mu(A) = \mu(-A)$ for all Borel
subsets $A \subset \RR^n$).

\begin{dfn*}
The $\Lambda_p$-body associated to $\mu$, for $p \geq 0$, is defined
as:
$$
\Lambda_p(\mu) := \set{\Lambda_{\mu} \leq p} \cap -\set{
\Lambda_{\mu} \leq p } ~.
$$
\end{dfn*}

\begin{lem} \label{eq_0000} \label{lem:Lambda-Zp}
Suppose $\mu$ is a log-concave probability measure on $\RR^n$ whose
barycenter lies at the origin. Then for any $p \geq 1$:
\[
\Lambda_p(\mu) \simeq p Z_p(\mu)^{\circ} ~.
\]
\end{lem}

These two equivalent points of view turn out to complement each
other well, and play a synergetic role in this work. Before
providing a proof, we illustrate this in the following naive
example. Given a log-concave probability measure $\mu$, a well known
consequence of Berwald's inequality (see e.g. \cite{MilmanPajor}) is
that:
\begin{equation} \label{eq:Zp-inclusion}
q \geq p \geq 1 \;\;\; \Rightarrow \;\;\; Z_p(\mu) \subset Z_q(\mu)
\subset C \frac{q}{p} Z_p(\mu) ~.
\end{equation}
In view of Lemma \ref{lem:Lambda-Zp}, note that this is nothing else
but a reformulation (up to constants) of the trivial set of
inclusions in (\ref{eq:Lambda-inclusion}).

\begin{proof}[Proof of Lemma \ref{lem:Lambda-Zp}]
First, suppose that $\xi \in \Lambda_p(\mu)$. Then:
$$
\int_{\RR^n} \exp(\abs{\scalar{\xi,x}}) d \mu(x) \leq \int_{\RR^n}
\exp(\scalar{\xi,x}) d \mu(x) + \int_{\RR^n} \exp(-\scalar{\xi,x}) d
\mu(x) \leq 2 e^p ~.
$$
Using the inequality $t^p / p! \leq e^t$, valid for any $t \geq 0$,
we see that:
$$
h_{Z_p(\mu)}(\xi) = \brac{\int_{\RR^n} \abs{\scalar{\xi,x}}^p d
\mu(x)}^{\frac{1}{p}} \leq \brac{2 e^p p!}^\frac{1}{p} \leq C p ~.
$$
Since $\xi \in \Lambda_p(\mu)$ was arbitrary, this amounts to
$\Lambda_p(\mu) \subseteq C p Z_p(\mu)^\circ$, the first desired
inclusion.

For the other inclusion, suppose $\xi \in \RR^n$ is such that
$h_{Z_p(\mu)}(\xi) \leq p$, that is:
\begin{equation}
\left(  \int_{\RR^n} \abs{\scalar{\xi,x}}^p d \mu(x) \right)^{1/p}
\leq p ~. \label{eq_540}
\end{equation}
Write $X$ for the random vector in $\RR^n$ that is distributed
according to $\mu$. Then the function:
$$ \varphi(t) = \PP(\scalar{X,\xi} \geq t) \quad \quad , \quad\quad t \in \RR ~, $$
is log-concave, according to the Pr\'ekopa-Leindler inequality (see,
e.g., the first pages of  \cite{Pis}). Furthermore, since the barycenter of $\mu$ lies at the origin, we have $1/e \leq
\varphi(0) \leq 1 - 1/e$ by the Gr\"unbaum--Hammer inequality (see e.g.
\cite[Lemma 3.3]{bobkov}). Using Markov's inequality, (\ref{eq_540})
entails that:
$$ \varphi(3 e p) \leq ( 3 e)^{-p} ~. $$
Since $\varphi$ is log-concave, then:
$$
\PP(\scalar{X,\xi} \geq t) = \varphi(t) \leq \varphi(0) \left( \frac{\varphi(3 ep)}{\varphi(0)}
\right)^{\frac{t}{3 ep}} \leq C \exp(-t / (3e)) \quad ,  \quad \forall t \geq 3 ep ~.
$$
An identical bound holds for $\PP(\scalar{X,\xi} \leq -t)$, and
combining the two, we obtain:
$$
\PP(|\scalar{X,\xi}| \geq t)  \leq C \exp(-t / (3e)) \quad ,  \quad \forall t \geq 3 ep ~.
$$
Therefore:
\begin{align*}
\EE \exp \left( \frac{|\scalar{\xi,X}|}{6e} \right) & = \frac{1}{6e}
\int_0^{\infty}
\exp \left( \frac{t}{6e}  \right) \PP(|\scalar{X,\xi}| \geq t)dt  \\
& \leq   \frac{1}{6e} \int_0^{3ep} \exp \left( \frac{t}{6e} \right) dt + C \int_{3
ep}^{\infty} \exp(-t / (6 e)) dt \leq \exp\left( \tilde{C} p
\right) ~.
\end{align*}
Consequently:
$$ \max \left \{ \Lambda_{\mu} \left(\frac{1}{6e} \xi \right),
\Lambda_{\mu} \left( -\frac{1}{6e} \xi \right) \right \} \leq \log
\EE \exp \left( \frac{|\scalar{\xi,X}|}{6e} \right) \leq C p ~,
$$
for some $C \geq 1$, and using (\ref{eq:Lambda-inclusion}), this
implies:
\[
\max \left \{ \Lambda_{\mu} \left(\frac{1}{6eC} \xi \right),
\Lambda_{\mu} \left( -\frac{1}{6eC} \xi \right) \right \} \leq p ~,
\]
for any $\xi \in \RR^n$ with $h_{Z_p(\mu)}(\xi) \leq p$. This is
precisely the second desired inclusion $p Z_p(\mu)^\circ \subseteq
C' \Lambda_p(\mu)$, and the assertion follows.
\end{proof}


\subsection{Level Sets of Convex Functions Under Gradient Maps}
The last topic we would like to review pertains to some properties
of level sets of convex functions and their gradient images. The
possibility to use the gradient image of $\Lambda_{\mu}$ as in
\cite{quarter} is one of the main reasons for additionally employing the logarithmic
Laplace transform, rather than working exclusively with the
$L_p$-centroid bodies.

\begin{lem}\label{lem:prod} 
Let $F: \RR^n \rightarrow \RR \cup
\{ \infty \}$ be a non-negative convex function, which is $C^1$-smooth in $\set{ F < \infty }$.
Let $q,r \geq 0$. Then:
\[
\scalar{z,\nabla F(x)} \leq q+r \quad \quad \quad \text{for any} \ z
\in \set{F \leq r}, \ x \in \frac{1}{2} \{ F \leq q \}.
\]
In other words:
\[
\nabla F \brac{\frac{1}{2} \set{F \leq q}} \subset (q+r) \set{F \leq
r}^{\circ} ~.
\]
\end{lem}
\begin{proof}
Since $F$ is non-negative and its graph lies above any tangent
hyperplane, then:
\[
\scalar{\nabla F(x),\frac{z}{2}} \leq F(x) + \scalar{\nabla F(x)
,\frac{z}{2}} \leq F(x+z/2) \leq \frac{F(2x) + F(z)}{2} \leq
\frac{q+r}{2} ~.
\]
\end{proof}

The following lemma was proved in \cite[Lemma 2.3]{K_psi2} for an
even function $F$.

\begin{lem} \label{lem:F-vr}
Let $F: \RR^n \rightarrow \RR \cup \{ \infty \}$ be a non-negative convex function,
$C^2$-smooth and strictly-convex in $\set{ F < \infty }$, with $F(0) = 0$. 
Let $p > 0$, and set:
$$
F_p := \set{F \leq p} \cap -\set{F \leq p} ~.
$$
Assume that:
\[
\Psi_p := \brac{\frac{1}{\vol(\frac{1}{2} F_p)} \int_{\frac{1}{2} F_p}
\det \Hess F(x) dx}^\frac{1}{n} > 0  ~.
\]
Then:
\[
\vr(F_p) \leq 2 \frac{\sqrt{p}}{\sqrt{\Psi_p}} ~.
\]
\end{lem}

\begin{proof}
Applying Lemma \ref{lem:prod} with $q=r=p$, and using the change of
variables $x = \nabla F(y)$, we obtain:
\[
\vol(2p (F_p)^\circ) \geq \vol \left(\nabla F \left(\frac{1}{2} F_p
\right) \right) = \int_{\frac{1}{2} F_p} \det \Hess F(y) dy = \vol
\left(\frac{1}{2} F_p \right) \Psi_p^n ~.
\]
Equivalently, we obtain:
\[
\vol((F_p)^\circ) \geq \brac{\frac{\Psi_p}{4p}}^{n} \vol(F_p) ~.
\]
Note that $F_p$ is a centrally-symmetric convex body, i.e., $F_p =
-F_p$. The Blaschke--Santal\'o inequality (see, e.g.,
\cite{Schneider-Book}) for a centrally-symmetric convex body $K$ asserts
that:
\[ 
\vr(K^\circ) \vr(K) \leq 1~.
\] 
Combining the last two estimates with $K = F_p$, the result
immediately follows.
\end{proof}


\section{A formula for $\vr(Z_p(\mu))$ involving tilted measures} \label{sec3}

Let $\mu$ denote a log-concave
probability measure on $\RR^n$ with density $\rho$, and let $\xi \in \set{\Lambda_{\mu} < \infty}$. We denote by $\mu_{\xi}$
the ``tilt'' of $\mu$ by $\xi$, defined via the following procedure. First, define the
probability density:
\[
\rho_\xi(x) := \frac{1}{Z_\xi} \rho(x) \exp(\scalar{\xi,x}) \quad
\quad \quad \text{for} \quad x \in \RR^n ~,
\]
where $Z_\xi > 0$ is a normalizing factor. Denoting by $b_\xi \in
\Real^n$ the barycenter of $\rho_\xi$, we set
$\mu_\xi$ to be the probability measure with density $\rho_\xi(\cdot -
b_\xi)$. Note that $\mu_\xi$ is a log-concave probability measure,
having the origin as its barycenter. Furthermore, as verified in
\cite[Section 2]{K_psi2}, we have:
\begin{equation}
 b_{\xi} = \nabla \Lambda_\mu(\xi) \quad  , \quad \cov(\mu_{\xi}) =
\Hess \Lambda_\mu(\xi) ~. \label{eq_2329}
\end{equation}

The following proposition is one of the main results in this
section:

\begin{prop} \label{prop:vr-formula}
Let $\mu$ denote a log-concave probability measure on $\Real^n$
whose barycenter lies at the origin. Then, for all $1 \leq p \leq
n$:
\begin{equation} \label{eq:vr-formula}
\vr(Z_p(\mu)) \simeq \sqrt{p} \inf_{x \in \frac{1}{2}
\Lambda_p(\mu)} \detcov(\mu_x)^{\frac{1}{2n}} .
\end{equation}
\end{prop}

In the proofs of the theorems stated in the Introduction, we will
not use the full force of Proposition \ref{prop:vr-formula}, but
rather only the lower bound for $\vr(Z_p(\mu))$. This lower bound
has a short proof, as the reader will see below. However, the
observation that we actually obtain an equivalence seems interesting, hence
we provide the arguments for both directions. Before going into the
proof, as a testament of its usefulness, we state the following
immediate corollary of Proposition \ref{prop:vr-formula}:

\begin{cor} Let $\mu$ be a log-concave probability measure on
$\RR^n$ whose barycenter lies at the origin. Then:
\[
1 \leq p \leq q \leq n \;\;\; \Rightarrow \;\;\;
\frac{\vr(Z_p(\mu))}{\sqrt{p}} \geq c \frac{\vr(Z_q(\mu))}{\sqrt{q}}
~.
\]
\end{cor}

\begin{rem} \label{rem:LYZ-1}
Using $q = n$ above and the fact that $\vr(Z_n(K)) \simeq \vr(K)$ for a convex body $K$ whose barycenter lies at the origin,
which follows from (\ref{eq:ZnK}) as in the Introduction, we immediately verify that:
\begin{equation} \label{eq:our-LYZ-K}
\forall 1 \leq p \leq n \;\;\; \vr(Z_p(K)) \geq c \sqrt{\frac{p}{n}}
\vr(K)~.
\end{equation}
This recovers up to a constant the lower bound of Lutwak, Yang and
Zhang (\ref{eq:LYZ-bound}). Moreover, recalling that $\vr(Z_n(\mu))
\simeq \sqrt{n} / \norm{\mu}_{L_\infty}^{\frac{1}{n}}$ by Lemma
\ref{lem:Zn} and the definition $(\ref{eq:Lmu})$ of $L_\mu$, the
same argument yields the following analog of (\ref{eq:our-LYZ-K}):
\[
\forall 1 \leq p \leq n \;\;\; \vr(Z_p(\mu)) \geq c \frac{\sqrt{p}}{L_\mu} \detcov(\mu)^{\frac{1}{2n}} = c \frac{\sqrt{p}}{L_\mu} \vr(Z_2(\mu)) ~.
\]
This may also be deduced by only employing the lower-bound in
(\ref{eq:vr-formula}), as in Remark \ref{rem:LYZ-2}.
\end{rem}

We now turn to the proof of Proposition \ref{prop:vr-formula}, and
begin with the lower bound for $\vr(Z_p(\mu))$. In fact, we show a
formally stronger statement:

\begin{lem} \label{lem_low}
Let $\mu$ denote a log-concave probability measure on $\Real^n$
whose barycenter lies at the origin. Then, for all $1 \leq p \leq
n$,
\[
\vr(Z_p(\mu)) \geq  c \sqrt{p} \sqrt{\Psi_p} ~,
\]
where $c > 0$ is a universal constant and:
\[
\Psi_p := \brac{\frac{1}{\vol(\frac{1}{2} \Lambda_p(\mu))} \int_{\frac{1}{2} \Lambda_p(\mu)} \detcov(\mu_x) dx}^\frac{1}{n} ~.
\]
\end{lem}
\begin{proof} Apply Lemma \ref{lem:F-vr} with $F = \Lambda_\mu$.
Since $\det \Hess \Lambda_\mu(x) = \detcov(\mu_x)$ according to
(\ref{eq_2329}), we deduce that:
\begin{equation} \label{eq_0005}
\vr(\Lambda_p(\mu)) \leq 2 \frac{\sqrt{p}}{\sqrt{\Psi_p}} ~.
\end{equation}
Applying Lemma \ref{eq_0000} in order to pass from
$\Lambda_p(\mu)$ to $Z_p(\mu)$, and the Bourgain--Milman
inequality (see, e.g., \cite{Pis}) for a centrally-symmetric
convex set $K \subset \RR^n$:
\[
\vr(K^\circ) \vr(K) \geq c~,
\]
we deduce from (\ref{eq_0005}) that:
\[
\vr(Z_p(\mu)) \simeq p \vr(\Lambda_p(\mu)^\circ) \gtrsim p
\vr(\Lambda_p(\mu))^{-1} \gtrsim \sqrt{p} \sqrt{\Psi_p} ~.
\]
\end{proof}

In order to deduce the upper bound of Proposition \ref{prop:vr-formula}, and
of crucial importance to the main results of this work, is the
following elementary observation:

\begin{prop} \label{prop:Lambda-iso}
Let $\mu$ denote a log-concave probability measure in $\Real^n$ with
barycenter at the origin. Then:
\[
 \forall x \in \frac{1}{2} \Lambda_p(\mu) \quad , \quad  \Lambda_p(\mu_x) \simeq \Lambda_p(\mu) ~.
\]
\end{prop}

Indeed, it is clear that the logarithmic Laplace transform should
commute nicely with the tilt operation, and the following identity
is verified by direct calculation:
\begin{equation} \label{eq:Lambda-tilt}
\Lambda_{\mu_x}(z) = \Lambda_\mu(z+x) - \Lambda_\mu(x) -
\scalar{z,b_x} ~, ~ b_x = \nabla \Lambda_\mu(x) ~.
\end{equation}
Geometrically, this means that the graph of $\Lambda_{\mu_x}$
is obtained from that of $\Lambda_\mu$ by subtracting the tangent
plane at $x$ (given by the linear function $z \mapsto \Lambda_\mu(x)
+ \scalar{z-x,\nabla \Lambda_\mu(x)}$), and translating everything by
$-x$ (so that $x$ gets mapped to the origin). In particular, we
verify that $\Lambda_{\mu_x}(0) = 0$ and that $\Lambda_{\mu_x} \geq
0$, as required from the logarithmic Laplace transform of a probability measure
with barycenter at the origin.

It remains to manipulate level sets of convex functions, once again. We
require the following:

\begin{lem} \label{lem:F-G}
Let $F$ be as in Lemma \ref{lem:prod}, and let $y \in \Real^n$ and $D, p > 0$. Define a function $G$ by:
\[
G(z) := F(z+y) - F(y) - \scalar{z,\nabla F(y)}~.
\]
Then:
\[
y \in \frac{1}{2} \set{F \leq D p }, \ z \in \set{F \leq p} \cap -\set{F \leq p} \quad \Longrightarrow \quad z \in 2 \set{G \leq
(D+1)p} ~.
\]
\end{lem}
\begin{proof} 
We apply Lemma \ref{lem:prod} with $q = Dp$ and $r=p$. Since $-z
\in \set{F \leq p}$ and $y \in \frac{1}{2} \set{F \leq D p}$, then by the conclusion of that lemma, $\scalar{-z,\nabla
F(y)} \leq (D+1)p$. Since $F$ is non-negative and convex, we deduce
that:
\[
G(z/2) \leq F(z/2+y) + \frac{D+1}{2} p \leq \frac{F(z) + F(2y)}{2} +
\frac{D+1}{2} p \leq (D+1) p ~.
\]
\end{proof}

\begin{proof}[Proof of Proposition \ref{prop:Lambda-iso}]
\hfill
\begin{enumerate}
\item If $z \in \Lambda_p(\mu)$, we apply Lemma \ref{lem:F-G} with $D=1$ and $y=x$ to $F = \Lambda_\mu$. By (\ref{eq:Lambda-tilt}), we deduce that $\Lambda_{\mu_x}(z/2) = G(z/2) \leq 2p$. Using (\ref{eq:Lambda-inclusion}), we conclude that $\Lambda_{\mu_x}(z/4) \leq p$. The same argument applies to $-z$ by the symmetry of our assumptions, and so we conclude that $z \in 4 \Lambda_p(\mu_x)$.

\item If $z \in \Lambda_p(\mu_x)$, we would like to apply Lemma \ref{lem:F-G} with $y=-x$ to $F = \Lambda_{\mu_x}$, since tilting $\mu_x$ by $-x$ gives back $\mu$. To this end, we must verify that $\Lambda_{\mu_x}(-2x) \leq D p$ for some $D>0$. According to (\ref{eq:Lambda-tilt}):
\[
\Lambda_{\mu_x}(-2x) = \Lambda_\mu(-x) - \Lambda_\mu(x) +
2\scalar{x,\nabla \Lambda_\mu(x)}~.
\]
By Lemma \ref{lem:prod}, we know that $\scalar{x,\nabla
\Lambda_\mu(x)} \leq 2p$, and using that $\Lambda_\mu$ is
non-negative, convex and vanishes at the origin, we obtain:
\[
\Lambda_{\mu_x}(-2x) \leq \frac{1}{2} \Lambda_\mu(-2x) + 4p \leq 4.5
p ~.
\]
We conclude that we may use $D=4.5$ above, and so Lemma
\ref{lem:F-G} finally implies that $\Lambda_{\mu}(z/2) = G(z/2) \leq
5.5 p$. As in the first part of the proof, we deduce that
$\Lambda_{\mu}(z/11) \leq p$. The same argument applies to $-z$ by
the symmetry of our assumptions, and so we conclude that $z \in
11 \Lambda_p(\mu)$.
\end{enumerate}
\end{proof}

Using Lemma \ref{eq_0000}, we equivalently reformulate Proposition
\ref{prop:Lambda-iso} as:

\begin{prop} \label{prop:Zp-iso}
Let $\mu$ denote a log-concave probability measure in $\Real^n$ with
barycenter at the origin. Then:
\[
 \forall x \in \frac{1}{2} \Lambda_p(\mu) \quad , \quad Z_p(\mu_x) \simeq Z_p(\mu) ~.
 \]
\end{prop}

To complete the proof of Proposition \ref{prop:vr-formula}, we state again Paouris' upper bound (\ref{eq:vr-upper}) on $\vr(Z_p(\nu))$:
\begin{prop}[Paouris] \label{prop:Paouris}
For any log-concave probability measure $\nu$ with barycenter at the
origin, and $2 \leq p \leq n$:
\[
\vr(Z_p(\nu)) \leq C \sqrt{p} \vr(Z_2(\nu))~.
\]
\end{prop}
\begin{proof}
The statement is invariant under linear transformations, so we may
assume that $\nu$ is isotropic. The claim is then the content of
\cite[Theorem 6.2]{Paouris-IsotropicTail}.
\end{proof}

\begin{proof}[Proof of Proposition \ref{prop:vr-formula}]
Lemma \ref{lem_low} implies the lower bound:
\[
\vr(Z_p(\mu)) \geq c \sqrt{p} \inf_{x \in
\frac{1}{2} \Lambda_p(\mu)} \detcov(\mu_x)^{\frac{1}{2n}}~.
\]
Since $\detcov(\mu_x)^{\frac{1}{2n}} = \vr(Z_2(\mu_x))$, then applying
Proposition \ref{prop:Paouris}, we obtain:
\begin{equation}
\inf_{x \in
\frac{1}{2} \Lambda_p(\mu)} \vr(Z_p(\mu_x)) \leq C \sqrt{p} \inf_{x \in \frac{1}{2}
\Lambda_p(\mu)} \detcov(\mu_x)^{\frac{1}{2n}}~. \label{eq_945}
\end{equation}
But by Proposition \ref{prop:Zp-iso}, $Z_p(\mu_x) \simeq Z_p(\mu)$ for all $x \in \frac{1}{2}
\Lambda_p(\mu)$, and hence the left-hand side in (\ref{eq_945}) is equivalent
to $\vr(Z_p(\mu))$, completing the proof.
\end{proof}

\begin{rem}
It follows that all of the inequalities which we used in the proof
of Proposition \ref{prop:vr-formula} above, are actually
equivalences up to numeric constants. This fact has some interesting consequences; we omit a detailed account of these here, and only remark on the following point. Given $1 \leq p \leq n$, denote:
\[
x_p := \text{argmin}_{x \in \frac{1}{2} \Lambda_p(\mu)} \detcov(\mu_x)^{\frac{1}{2n}} ~,
\]
so that $\mu_{x_p}$ is the ``worst" tilt we need to account for when evaluating $\vr(Z_p(\mu))$. It follows that for this tilt:
\[
\vr(Z_p(\mu_{x_p})) \simeq \sqrt{p} \vr(Z_2(\mu_{x_p})) ~,
\]
and in particular, the argument described in Subsection
\ref{subsec:slicing} implies that $L_{\mu_{x_p}} \leq C \sqrt{n /
p}$. It is interesting to compare this with the approach from
\cite{quarter} for resolving the isomorphic slicing problem. The
latter approach is in some sense dual to our current one, since in
this work our goal will be to bound $\detcov(\mu_{x_p})^{\frac{1}{2n}}$ from
below, whereas the goal in \cite{quarter} was to bound this
expression from above. Compare also with Remark \ref{rem:LYZ-2}.
\end{rem}


\section{On Paouris' definition of $q^*$}
\label{sec4}

Given a centrally-symmetric convex body $K \subset \Real^n$, its
``(dual) Dvoretzky-dimension" $k^*(K)$ was defined by V. Milman and
G. Schechtman \cite{MilmanSchechtmanSharpDvorDim} as the largest
positive integer $k \leq n$ so that:
\[
\sigma_{n,k}\set{E \in G_{n,k} ; \frac{1}{2} W(K) B_E \subset Proj_E
K \subset 2 W(K) B_E } \geq \frac{n}{n+k} ~,
\]
where $\sigma_{n,k}$ denotes the Haar probability measure on
$G_{n,k}$ and $B_E$ denotes the Euclidean unit ball in the subspace
$E$. It was shown in \cite{MilmanSchechtmanSharpDvorDim}, following
Milman's seminal work \cite{Mil71}, that:
\begin{equation} \label{eq:k*}
k^*(K) \simeq n \brac{\frac{W(K)}{\diam(K)}}^2 ~.
\end{equation}
Define $W_q(K) = \left( \int_{S^{n-1}} h_K(\theta)^q d
\sigma(\theta) \right)^{\frac{1}{q}}$, the $q$-th moment of the supporting
functional of $K$. According to Litvak, Milman and Schechtman
\cite{LMS}:
\begin{equation} \label{eq:LMS1}
c_1 W_q(K) \leq \max \left \{ W(K) , \sqrt{q/n} \; \;
\diam(K) \right \} \leq c_2 W_q(K) ~.
\end{equation}
The quantity $W_q(Z_q(\mu))$ has a simple equivalent description: a
direct calculation as in \cite{Paouris-Small-Diameter} confirms that
for any Borel probability measure $\mu$ on $\RR^n$ and $q \geq 1$:
\begin{equation} \label{WQZQ}
W_q(Z_q(\mu)) \simeq
\frac{\sqrt{q}}{\sqrt{n+q}} I_q(\mu) ~ ~ , ~ ~ I_q(\mu) := \brac{ \int_{\Real^n} |x|^q d\mu(x) }^{\frac{1}{q}} ~.
\end{equation}
Finally, observe that when the barycenter of $\mu$ is at the origin, then $I_2(\mu)^2 = \trcov(\mu)$.

\medskip

In \cite{Paouris-IsotropicTail}, Paouris defines $q^*(\mu)$ as
follows:
\[
q^*(\mu) := \sup\set{q \in \mathbb{N} ; k^*(Z_q(\mu)) \geq q} ~.
\]
It is straightforward to check that all of Paouris' results
involving $q^*(\mu)$ from
\cite{Paouris-IsotropicTail,PaourisSmallBall} remain valid when
replacing it with $q^*_c(\mu)$ when $c > 0$  is a fixed universal
constant, where $q^*_\delta$ is defined as follows:
\[
q^*_\delta(\mu) := \sup\set{q \geq 1 ; k^*(Z_q(\mu)) \geq
\delta^{-2} q} ~.
\]
Although the particular value of $c>0$ seems insignificant for the
results of \cite{Paouris-IsotropicTail,PaourisSmallBall},
the definition we require in this work is essentially that of $q^*_c$ for some \emph{small enough} universal
constant $c>0$. Our preference to work with a variant of $q^*_c$ is motivated by Lemma \ref{lem:P3-equiv} below and
the subsequent remarks.

\medskip

We proceed as follows. Given a log-concave probability measure $\mu$
on $\RR^n$, $q \geq 1$ and $\delta>0$, consider the following four
related properties:
\begin{enumerate}
\item $P_1(\delta)$ is the property that $k^*(Z_q(\mu)) \geq \delta^{-2} q$.
\item $P_1'(\delta)$ is the property that $\diam(Z_q(\mu)) \leq \delta \sqrt{n} \frac{W(Z_q(\mu))}{\sqrt{q}}$.
\item $P_2(\delta)$ is the property that $\diam(Z_q(\mu)) \leq \delta \sqrt{n} \detcov(\mu)^{\frac{1}{2n}}$.
\item $P_W$ is the property that $W(Z_q(\mu)) \geq c \sqrt{q} \detcov(\mu)^{\frac{1}{2n}}$, for some specific, appropriately small universal constant $c>0$, as in the proof of Lemma \ref{lem:P3-equiv}(2) below.
\end{enumerate}

According to (\ref{eq:k*}), we have:
\begin{equation} \label{eq:P_1-P_1'}
P_1(\delta) \Rightarrow P_1'(C_1 \delta) \Rightarrow P_1(C_2 \delta)
~,
\end{equation}
for all $\delta > 0$, where $C_1, C_2 > 1$ are universal constants. The next lemma relates between the other
properties above:

\begin{lem} \label{lem:P3-equiv} Suppose $\mu$ is a log-concave probability measure in $\RR^n$ whose barycenter lies at the origin.
Let $q \in [1,n]$ and $\delta \in (0,1]$. Then:

\begin{enumerate}
\item If $\mu$ is isotropic and $P_1(\delta)$ holds, then $P_2(C_3 \delta)$ holds.
\item
\begin{enumerate}
\item If $P_1'(\delta)$ holds, then so does $P_W$.
\item Suppose $\delta < \delta_0$ for a certain appropriately small universal constant $\delta_0 > 0$. \\ If $P_2(\delta)$ holds, then so does $P_W$.
\end{enumerate}
\item If $P_2(\delta)$ and $P_W$ hold, then so does $P_1'(C_4 \delta)$.
\end{enumerate}
\end{lem}

\begin{proof} \hfill
\begin{enumerate}
\item
Clearly $P_1(\delta)$ implies $P_1(1)$. Using (\ref{WQZQ}), Paouris's main result \cite[Theorem 8.1]{Paouris-IsotropicTail} and the isotropicity of $\mu$, we know that:
\[
W_q(Z_q(\mu)) \simeq \frac{\sqrt{q}}{\sqrt{n}} I_q(\mu) \simeq \frac{\sqrt{q}}{\sqrt{n}} I_2(\mu) = \frac{\sqrt{q}}{\sqrt{n}} \brac{\trcov(\mu)}^{\frac{1}{2}} = \sqrt{q} ~.
\]
In particular, $W(Z_q(\mu)) \leq W_q(Z_q(\mu)) \leq C \sqrt{q}$. Since
$P_1(\delta)$ implies $P_1'(C_1 \delta)$, then:
\[
\diam(Z_q(\mu)) \leq C_1 \delta \sqrt{n} \frac{W(Z_q(\mu))}{\sqrt{q}} \leq C C_1 \delta \sqrt{n} = C_3 \delta \sqrt{n} \detcov(\mu)^{\frac{1}{2n}}  ~,
\]
and $P_2(C_3 \delta)$ holds true.
\item
Since all properties are invariant under scaling, we may assume that $\detcov(\mu) = 1$.
Using (\ref{WQZQ}) and the arithmetic-geometric mean inequality:
\[
\frac{1}{n} I_2(\mu)^2 = \frac{1}{n} \trcov(\mu) \geq \detcov(\mu)^{\frac{1}{n}} ~,
\]
we see that:
\begin{equation} \label{eq_330}
W_q(Z_q(\mu)) \geq c_0 \frac{\sqrt{q}}{\sqrt{n}} I_q(\mu) \geq c_0 \frac{\sqrt{q}}{\sqrt{n}} I_2(\mu) \geq c_0 \sqrt{q} ~.
\end{equation}
\begin{enumerate}
\item
Assuming $P_1'(\delta)$, (\ref{eq:LMS1}) implies that $W(Z_q(\mu)) \geq c_1 W_q(Z_q(\mu))$, and together with (\ref{eq_330}), $P_W$ follows.
\item
Set $\delta_0 = c_0 c_1$, where $c_0$ is the constant from (\ref{eq_330}) and $c_1$ is the constant from (\ref{eq:LMS1}).
Using (\ref{eq_330}), the property $P_2(\delta)$ with $0 < \delta < \delta_0$ implies:
$$
\frac{\sqrt{q}}{\sqrt{n}} \diam(Z_q(\mu))
\leq \delta \sqrt{q} < c_0 c_1 \sqrt{q} \leq c_1 W_q(Z_q(\mu)) ~.
$$
Therefore by (\ref{eq:LMS1}), $W(Z_q(\mu)) \geq c_1 W_q(Z_q(\mu))
\geq c_0 c_1 \sqrt{q}$, and $P_W$ follows.
\end{enumerate}
\item
This is immediate by plugging the estimates on $\diam(Z_q(\mu))$ and $W(Z_q(\mu))$ into the definition of $P_1'(\delta)$. 
\end{enumerate}
\end{proof}

\begin{rem}
Inspecting the proof, one may check that the assumption that $\delta \leq 1$ is not essential
for the proof of parts (1), (2a) and (3), if one allows
different dependence on $\delta$ in the conclusion of the
assertions. However, the assumption that $\delta < \delta_0$ was \emph{crucially} used in
the proof of part (2b).
\end{rem}

We conclude from Lemma \ref{lem:P3-equiv} and
(\ref{eq:P_1-P_1'}) that $P_1(\delta)$ implies all the other
properties if $\mu$ is isotropic, and that $P_2(\delta)$ implies all
the other properties if $\delta$ is small enough. Neither of these
restrictions are essential for the purposes of this work, but
nevertheless we prefer to proceed with the more accessible
$P_2(\delta)$ property, since in addition and in contrast to the
$P_1(\delta)$ one, it is more stable in the following sense:
\begin{enumerate}
\item For any $\mu$, if $P_2(\delta)$ holds for $q$, then it also holds for all $p$ with $1 \leq p < q$.
\item If $\mu$ is isotropic and $P_2(\delta)$ holds for $\mu$ with $q$, then $P_2(\delta \sqrt{n / k})$ holds for
$\pi_E \mu$ with $q$, simply because $Z_q(\pi_E \mu) = Proj_E
Z_q(\mu)$ for all $E \in G_{n,k}$.
\end{enumerate}

Consequently, we make the following:
\begin{dfn*}
\begin{eqnarray*}
q^{\sharp}(\mu) & := & \sup \set{q \geq 1 \; ; \; \diam(Z_q(\mu)) \leq c^\sharp \sqrt{n} \detcov(\mu)^{\frac{1}{2n}}} \\
& = & \Delta_\mu^{-1}(c^\sharp \sqrt{n} \detcov(\mu)^{\frac{1}{2n}})
~,
\end{eqnarray*}
where $[1,\infty) \ni q \mapsto \Delta_\mu(q) := \diam(Z_q(\mu))$
and $c^\sharp>0$ is a small enough constant, to be prescribed in
Lemma \ref{lem:q-sharp} below. \\
As a convention, if $\diam(Z_1(\mu)) \geq c^\sharp \sqrt{n} \detcov(\mu)^{\frac{1}{2n}}$, we set $q^\sharp(\mu) = 1$.
\end{dfn*}

\begin{lem} \label{lem:q-sharp}
We may choose the numeric constant $c^\sharp > 0$ small enough so
that:
\begin{enumerate}
\item $q^\sharp(\mu) \leq n$.
\item $1 \leq q \leq q^\sharp(\mu)$ implies $k^*(Z_q(\mu)) \geq q$ and $W(Z_q(\mu)) \geq c \sqrt{q} \detcov(\mu)^{\frac{1}{2n}}$.
\end{enumerate}
\end{lem}
\begin{proof}
Assume first that $q^\sharp(\mu) > 1$. The second point follows immediately from Lemma \ref{lem:P3-equiv}
and (\ref{eq:P_1-P_1'}). The first point follows from (\ref{WQZQ}), since:
\[
n \cdot \detcov(\mu)^{\frac{1}{n}} \leq \trcov(\mu) = I_2(\mu)^2 \leq I_n(\mu)^{2} \simeq W_n(Z_n(\mu))^2 \leq \diam(Z_n(\mu))^2 ~.
\]
It remains to deal with the degenerate case $q^\sharp(\mu) = 1$. By definition, $k^*(Z_1(\mu)) \geq 1$, and e.g. by (\ref{eq:k*}):
$$
 W(Z_1(\mu)) \geq c \frac{\diam(Z_1(\mu))}{\sqrt{n}} \geq  c c^\sharp \; \detcov(\mu)^{\frac{1}{2n}} ~,
$$
as required.
\end{proof}

Consequently $\lfloor q^\sharp(\mu) \rfloor \leq q^*(\mu)$, and all
of Paouris' results for $q \leq q^*(\mu)$ continue to hold for $q
\leq q^{\sharp}(\mu)$. Similarly, by Lemma \ref{lem:P3-equiv}, if
$\mu$ is isotropic then $q^*_c(\mu) \leq q^\sharp(\mu)$ for some
small constant $c>0$.
To conclude this section, we reiterate the stability of
$q^\sharp(\mu)$ under projections in the following corollary, which
is one of the key ingredients in the proof of Theorem
\ref{main_thm3}:

\begin{cor} \label{cor:q-sharp-proj}
Let $\mu$ denote an isotropic log-concave probability measure in
$\Real^n$, let $1 \leq k \leq n, q \geq 1$. Then for all $E \in
G_{n,k}$ with $k \geq (c^{\sharp})^{-2} \diam^2(Z_q(\mu))$, we have
$q^\sharp(\pi_E \mu) \geq q$. In particular $k^*(Proj_E Z_q(\mu))
\geq q$ and $W(Proj_E Z_q(\mu)) \geq c \sqrt{q}$.
\end{cor}
\begin{proof}
Since $\pi_E \mu$ remains isotropic, $Z_q(\pi_E \mu) = Proj_E
Z_q(\mu)$ and $\diam(Proj_E Z_q(\mu)) \leq \diam(Z_q(\mu)) \leq
c^\sharp \sqrt{k}$, the assertion follows by definition of
$q^\sharp(\pi_E \mu)$ and Lemma \ref{lem:q-sharp}.
\end{proof}

\section{Controlling $\detcov(\mu_x)$ via projections} \label{sec5}

In view of Proposition \ref{prop:vr-formula}, our goal now is to bound from below $\detcov(\mu_x)^{\frac{1}{2n}}$
for the tilted measures $\mu_x$, where $x \in \frac{1}{2}
\Lambda_p(\mu)$. Our only available information is given by Proposition \ref{prop:Zp-iso}, stating that $Z_p(\mu_x) \simeq Z_p(\mu)$, where
$\mu$ itself is assumed isotropic.

\subsection{Finding a single good direction}

Suppose $\nu$ is a log-concave probability measure on $\Real^n$ whose
barycenter lies at the origin. Recall that its isotropic constant is defined as:
\begin{equation} \label{eq:L_nu}
L_\nu := \norm{\nu}_{L_\infty}^{\frac{1}{n}}
\detcov(\nu)^{\frac{1}{2n}} ~.
\end{equation}
Since the isotropic constant $L_\nu$ satisfies $L_\nu \geq c > 0$
(see e.g. \cite{MilmanPajor,K_psi2}), then according to Lemma
\ref{lem:Zn}:
\begin{equation} \label{eq:basic-est}
\detcov(\nu)^{\frac{1}{2n}} \gtrsim \frac{1}{\norm{\nu}_{L_\infty}^{\frac{1}{n}}} \simeq \frac{\vr(Z_n(\nu))}{\sqrt{n}}  ~.
\end{equation}

%
%
\begin{rem} \label{rem:LYZ-2}
Since $Z_n(\mu_x) \simeq Z_n(\mu)$ whenever $x \in
\frac{1}{2} \Lambda_n(\mu)$, we immediately see by (\ref{eq:basic-est})
and (\ref{eq:L_nu}) that in this case:
\[
\detcov(\mu_x)^{\frac{1}{2n}} \gtrsim \frac{\vr(Z_n(\mu_x))}{\sqrt{n}}
\simeq \frac{\vr(Z_n(\mu))}{\sqrt{n}} \simeq \frac{1}{\norm{\mu}_{L_\infty}^{\frac{1}{n}}} \simeq \frac{\detcov(\mu)^{\frac{1}{2n}}}{L_\mu} ~,
\]
as already noted in \cite[Formula (50)]{K_psi2}. Using the lower bound on $\vr(Z_p(\mu))$ given by Lemma \ref{lem_low},
it follows that:
\[
\vr(Z_p(\mu)) \gtrsim \frac{\sqrt{p}}{L_\mu} \vr(Z_2(\mu)) \quad , \quad \forall 1 \leq p \leq n ~,
\]
recovering the extended Lutwak--Yang--Zhang lower-bound from Remark \ref{rem:LYZ-1}.
This however misses our goal in this section by a factor of $L_\mu$.
\end{rem}

We next generalize the basic estimate (\ref{eq:basic-est}) to handle
other (say integer) values of $k$ between $1$ and $n$, by projecting
onto a lower dimensional subspace:

\begin{lem} \label{lem:basic-est2}
Let $\nu$ denote a log-concave probability measure in $\Real^n$ with
barycenter at the origin, and let $k$ denote an integer between $1$
and $n$. Then:
\begin{equation} \label{eq:basic-est2}
\exists \theta \in S^{n-1} \;\;\; \sqrt{ \int_{\RR^n} \scalar{x, \theta}^2 d \nu(x) } \geq
\frac{c}{\sqrt{k}} \sup_{E \in G_{n,k}} \vr(Proj_E Z_k(\nu)) ~.
\end{equation}
\end{lem}
\begin{proof}
Given $E \in G_{n,k}$, apply (\ref{eq:basic-est}) to $\pi_E \nu$ and
note that $Z_k(\pi_E \nu) = Proj_E Z_k(\nu)$.
\end{proof}

The idea now is to compare $\vr(Proj_E Z_k(\mu_x))$ with $\vr(Proj_E
Z_k(\mu))$. Note that if $Z_p(\nu) \simeq Z_p(\mu)$,
then by (\ref{eq:Zp-inclusion}):
\[
1 \leq q \leq p \;\;\; \Rightarrow \;\;\; c \frac{q}{p} Z_q(\mu)
\subset Z_q(\nu) \subset C \frac{p}{q} Z_q(\mu) ~,
\]
and so $\vr(Proj_E Z_k(\nu)) \geq c \frac{k}{p} \vr(Proj_E
Z_k(\mu))$ for all $E \in G_{n,k}$, whenever $k \leq p$. To control $\vr(Proj_E
Z_k(\mu))$, we have:

\begin{lem} \label{lem:controling-vr}
Let $\mu$ denote a log-concave probability measure in $\Real^n$ with
barycenter at the origin, and let $1 \leq k \leq q^\sharp(\mu)$.
Then:
\[
\exists E \in G_{n,k} \;\;\; \vr(Proj_E Z_k(\mu)) \geq c \sqrt{k}
\detcov(\mu)^{\frac{1}{2n}} ~.
\]
\end{lem}
\begin{proof}
Lemma \ref{lem:q-sharp} asserts that $1 \leq k \leq q^\sharp(\mu)$
implies that $k^*(Z_k(\mu)) \geq k$. Consequently, there exists at
least one (in fact, many) $E \in G_{n,k}$ so that:
\[
\frac{1}{2} W(Z_k(\mu)) B_E \subset Proj_E Z_k(\mu) \subset 2
W(Z_k(\mu)) B_E ~,
\]
and hence $\vr(Proj_E Z_k(\mu)) \geq \frac{1}{2} W(Z_k(\mu))$. It
remains to appeal to Lemma \ref{lem:q-sharp} again and deduce from
$1 \leq k \leq q^\sharp(\mu)$ that $W(Z_k(\mu)) \geq c \sqrt{k}
\detcov(\mu)^{\frac{1}{2n}}$.
\end{proof}

Combining all of the preceding discussion, we obtain the
following fundamental:

\begin{prop} \label{prop:det-basic-est}
Let $\nu,\mu$ denote two log-concave probability measures in
$\Real^n$ with barycenters at the origin, and let $1 \leq p \leq n$.
Assume that $Z_p(\nu) \simeq Z_p(\mu)$. Then:
\[
\exists \theta \in S^{n-1} \;\;\; \sqrt{ \int_{\RR^n} \scalar{x,
\theta}^2 d \nu(x) } \geq 
c \min \left \{
1,\frac{q^\sharp(\mu)}{p} \right\} \detcov(\mu)^{\frac{1}{2n}} ~.
\]
\end{prop}

\begin{rem}
To avoid ambiguity of our notation, we explicitly remark that throughout this section, all statements which \emph{assume} that $Z_p(\nu) \simeq Z_p(\mu)$, in
fact apply whenever $\frac{1}{B} Z_p(\mu) \subseteq Z_p(\nu) \subseteq B Z_p(\mu)$ for \emph{any} parameter $B \geq 1$, with the resulting constants in the conclusion of those statements depending in addition on $B$.
\end{rem}

\subsection{Controlling the entire $\detcov(\nu)$}

We can now proceed to control the entire $\detcov(\nu)$ by
projecting onto the flag of subspaces spanned by the eigenvectors of
$\cov(\nu)$. To apply Proposition \ref{prop:det-basic-est}, we
require good control over $q^\sharp(\pi_E \mu)$. One way to obtain
such control is to make a definition:

\begin{dfn*}
The Hereditary-$q^\sharp$ constant of a log-concave probability
measure $\mu$ on $\Real^n$, denoted $q^\sharp_H(\mu)$, is defined
as:
\[
q^\sharp_H(\mu) := n \; \inf_k \inf_{E \in G_{n,k}}
\frac{q^\sharp(\pi_E \mu)}{k}
~.
\]
\end{dfn*}

\begin{rem} \label{rem:qH-iso}
It is useful to note the following alternative formula for $q^\sharp_H(\mu)$, valid only for
an {\it isotropic}, log-concave probability measure $\mu$ on
$\RR^n$. Recalling the definitions of $q^\sharp(\nu)$, $\Delta_\nu(q) = \diam(Z_q(\nu))$,
and using $\sup_{E \in G_{n,k}} \diam(Proj_E Z_q(\mu)) =
\diam(Z_q(\mu))$, we obtain:
\begin{equation} \label{eq:q_H-extended}
q^\sharp_H(\mu) =  n \; \inf_{1 \leq k \leq n}
\frac{\Delta_\mu^{-1}(c^\sharp \sqrt{k})}{k} \simeq
 n \; \inf_{1 \leq q \leq q^\sharp(\mu)} \frac{q}{\diam(Z_q(\mu))^2} ~,
\end{equation}
where we use (\ref{eq:Zp-inclusion}) and our convention for when
$q^\sharp(\nu) = 1$ to justify the last equivalence.
\end{rem}

\begin{prop} \label{prop:Her}
Let $\nu,\mu$ denote two log-concave probability measures in
$\Real^n$ with barycenters at the origin, and assume that $\mu$ is
isotropic. Let $1 \leq p \leq A q^\sharp_H(\mu)$ with $A \geq 1$, and assume that $Z_p(\nu) \simeq Z_p(\mu)$.
Then:
\[
\detcov(\nu)^{\frac{1}{2n}} \geq \frac{c}{A} ~,
\]
where $c > 0$ denotes a universal constant.
\end{prop}
\begin{proof} 
Let $0 < \lambda_1 \leq \lambda_2 \leq \ldots \leq \lambda_n$
denote the eigenvalues of $\cov(\nu)$, and let $E_k \in G_{n,k}$
denote the subspace spanned by the eigenvectors corresponding to
$\lambda_1,\ldots,\lambda_k$. Since $Proj_{E_k} Z_p(\nu) \simeq Proj_{E_k} Z_p(\mu)$, Proposition
\ref{prop:det-basic-est} applied to $\pi_{E_k} \nu$ and $\pi_{E_k}
\mu$ implies that:
\[
\sqrt{\lambda_k} \geq c \min\brac{1,\frac{q^\sharp(\pi_{E_k}
\mu)}{p}} \geq c \min\brac{1,\frac{q^\sharp_H(\mu)}{p} \frac{k}{n}}
\geq \frac{c}{A} \frac{k}{n} ~.
\]
Taking geometric average over the $\lambda_k$'s, the assertion
immediately follows.
\end{proof}

\begin{rem} \label{rem:geom-avg}
It is clear from the proof that we may actually replace in the definition of $q_H^\sharp(\mu)$ the infimum over $k$
with a geometric-average over the terms. For future reference, we denote this variant by $q^\sharp_{GH}(\mu)$, and as in Remark \ref{rem:qH-iso}, obtain the following expression for it when $\mu$ is in addition \emph{isotropic}:
\begin{equation} \label{eq:qGH}
q^\sharp_{GH}(\mu) = n \; \brac{\prod_{k=1}^n
\frac{\Delta_\mu^{-1}(c^\sharp \sqrt{k})}{k}}^{\frac{1}{n}} \simeq
\brac{\prod_{k=1}^n \Delta_\mu^{-1}(c^\sharp
\sqrt{k})}^{\frac{1}{n}} ~.
\end{equation}
\end{rem}

Another way to obtain some (partial) control over $q^\sharp(\pi_E
\mu)$ is to invoke Corollary \ref{cor:q-sharp-proj}:

\begin{prop} \label{prop:log}
Let $\nu,\mu$ denote two log-concave probability measures in
$\Real^n$ with barycenters at the origin, and assume that $\mu$ is
isotropic. Let $1 \leq p \leq n$ and $A \geq 1$. Assume that $Z_p(\nu) \simeq Z_p(\mu)$ and
that:
\begin{equation} \label{eq:diam-log-assump}
\diam(Z_p(\mu)) \sqrt{\log(p)} \leq A \sqrt{n} ~.
\end{equation}
Then:
\[
\detcov(\nu)^{\frac{1}{2n}} \geq \exp(-C A^2) ~.
\]
\end{prop}

\begin{proof}
We employ the same notation as in the previous proof. Setting:
\[
k_0 := \lceil (c^{\sharp})^{-2} \diam^2(Z_p(\mu)) \rceil ~,
\]
Corollary \ref{cor:q-sharp-proj} states  that
$q^\sharp(\pi_{E_{k_0}} \mu) \geq p$. Consequently, applying
Proposition \ref{prop:det-basic-est} to $\pi_{E_{k_0}} \nu$ and
$\pi_{E_{k_0}} \mu$, we obtain that $\lambda_{k_0} \geq c > 0$, and
hence the largest $n-k_0+1$ eigenvalues of $\cov(\nu)$ are bounded
below by the same $c > 0$. To bound the contribution of the other
eigenvalues, we use (\ref{eq:Zp-inclusion}) to obtain the following trivial bound (which may be
improved, but ultimately only results in better numeric constants):
\begin{eqnarray*}
\detcov(\pi_{E_{k_0}} \nu)^{\frac{1}{2k_0}} &=& \vr(Z_2(\pi_{E_{k_0}} \nu)) \gtrsim \frac{1}{p} \vr(Z_p(\pi_{E_{k_0}} \nu)) \\
&\simeq& \frac{1}{p} \vr(Z_p(\pi_{E_{k_0}} \mu)) \geq \frac{1}{p}
\vr(Z_2(\pi_{E_{k_0}} \mu)) = \frac{1}{p} ~.
\end{eqnarray*}
Using our estimates separately on $E_{k_0}$ and $E_{k_0}^\perp$, we
obtain:
\[
\detcov(\nu)^{\frac{1}{2n}} =  \left(\detcov(\pi_{E_{k_0}} \nu)
\detcov(\pi_{E_{k_0}^\perp} \nu) \right)^{\frac{1}{2n}} \geq c
\brac{\frac{1}{p}}^{\frac{k_0}{n}} ~.
\]
Our assumption (\ref{eq:diam-log-assump}) precisely ensures that
$k_0 \log(p) \leq C \cdot A^2 n$, and the assertion follows.
\end{proof}

\begin{rem}
Our choice of working in this section with $q^{\sharp}(\mu)$ instead
of $q^*_c(\mu)$ is only a matter of convenience and is not of
essence, as justified in Section \ref{sec4}.
\end{rem}

\subsection{Proofs of Main Theorems} \label{subsec:generalizations}

Theorem \ref{main_thm3} now follows immediately from
Proposition \ref{prop:log}, combined with
Propositions
\ref{prop:vr-formula} and \ref{prop:Zp-iso}.
Similarly, Proposition \ref{prop:Her} and Remark \ref{rem:geom-avg}, combined with
Propositions \ref{prop:vr-formula} and \ref{prop:Zp-iso}, yield:

\begin{thm} \label{main_thm2+}
Let $\mu$ denote an isotropic log-concave probability measure in
$\Real^n$. Then:
\[
\vr(Z_p(\mu)) \geq c \sqrt{p} \quad \quad , \quad \quad \forall \; 2 \leq p \leq C q^\sharp_H(\mu) ~.
\]
Moreover, the same bound remains valid for $2 \leq p \leq C q^\sharp_{GH}(\mu)$.
\end{thm}

Now if $\mu$ is a log-concave isotropic measure on $\Real^n$ which is in addition a $\psi_\alpha$-measure with constant $b_\alpha$ (for $\alpha \in [1,2]$), by definition:
\[
\diam(Z_p(\mu)) \leq 2 b_\alpha p^{\frac{1}{\alpha}} ~.
\]
It therefore follows immediately from (\ref{eq:q_H-extended}) that:
\[
q^\sharp_H(\mu) \geq \frac{c}{b_\alpha^\alpha} n^{\alpha/2} ~,
\]
and thus Theorem \ref{main_thm2} follows from Theorem \ref{main_thm2+}.

\medskip

Lastly, it may be worthwhile to record the following generalization of Theorems \ref{main_thm} and \ref{main_thm1+}, which follows
immediately, as in Subsection \ref{subsec:slicing}, from Theorem \ref{main_thm2+} and (\ref{eq:qGH}):
\begin{thm}
Let $\mu$ denote a log-concave probability measure in $\Real^n$ with barycenter at the origin. Then:
\[
L_\mu \leq C \brac{\prod_{k=1}^n \frac{k}{\Delta_\mu^{-1}(c^\sharp
\sqrt{k})}}^{\frac{1}{2n}} ~.
\]
\end{thm}
\noindent
Observe that in this formulation, we only require an \emph{on-average} control over the growth of $\Delta_\mu(p) = \diam(Z_p(\mu))$, as opposed to all previously mentioned bounds on $L_\mu$.


\section{Equivalence to the Slicing Problem} \label{sec_counter} \label{sec6}

Denote:
\begin{equation} \label{eq_2246}
 L_n := \sup_{K \subseteq \RR^n} L_K  ~,
\end{equation}
where the supremum runs over all convex bodies $K \subset \RR^n$.
Recall that $K$ is called isotropic if $\mu_K$, the uniform measure
on $K$, is isotropic. Recall also  that $Z_p(K) = Z_p(\mu_K)$. In
this section, we observe that removing the logarithmic factor in
Theorem \ref{main_thm3} is in fact equivalent to Bourgain's
hyperplane conjecture.

\begin{thm} \label{thm:SlicingEquiv}
Given $n \geq 1$, the following statements are equivalent:
\begin{enumerate}
\item There exists $A>0$ so that $L_n \leq A$.
\item There exists $B>0$ so that for any isotropic convex body $K \subset \Real^n$, we have:
\begin{equation} \label{eq:SlicingEquivAssumption}
\vr(Z_p(K)) \geq \sqrt{p} / B \quad \quad \forall 1 \leq p \leq q^{\sharp}(\mu_K) / B ~.
\end{equation}
\end{enumerate}
The equivalence is in the sense that the parameters $A,B$ above depend solely one on the other, and not on the dimension $n$.
\end{thm}

The proof is based on the following construction from Bourgain,
Klartag and Milman \cite{BKM}. Given $m \geq 1$, let $K_m$ denote an
isotropic convex body with $L_{K_m} \geq c L_m$. Choosing $c>0$
appropriately, it is well-known (see, e.g., the last remark in
\cite{K_jfa}) that we may assume that $K_m$ is centrally-symmetric
and satisfies $K_m \subset 10 \sqrt{m} B_m$. We also set $D_m :=
\sqrt{m+2} B_m$, and note that $D_m$ is isotropic. Given $1/n \leq
\lambda < 1$, consider the cartesian product:
$$
T_{\lambda} = K_{\lfloor \lambda n \rfloor} \times D_{\lceil (1-\lambda) n \rceil} \subseteq \RR^n ~.
$$
Clearly, $T_{\lambda}$ is a centrally-symmetric isotropic convex body, and since $L_{D_m} \simeq 1$, it follows that:
\begin{equation} \label{eq:same-LK}
L_{T_{\lambda}} \simeq L_{\lfloor \lambda n \rfloor}^{\lfloor
\lambda n \rfloor / n} \simeq L_{\lfloor \lambda n \rfloor}^\lambda
~.
\end{equation}

\begin{lem} \label{lem_2335}
For any pair of centrally-symmetric convex bodies $K_1 \subset \RR^{n_1}, K_2 \subset \RR^{n_2}$ and $p \geq 1$, we have:
\[
\frac{1}{2} (Z_p(K_1) \times Z_p(K_2)) \subset Z_p(K_1 \times K_2) \subset Z_p(K_1) \times Z_p(K_2) ~.
\]
\end{lem}

\begin{proof}
Denote $E_1 := \Real^{n_1} \times \set{0}$ and $E_2 := \set{0}
\times \Real^{n_2}$. By definition, $Z_p(K_1 \times K_2) \cap E_1 =
Z_p(K_1) \times \set{0}$ and $Z_p(K_1 \times K_2) \cap E_2  =
\set{0} \times Z_p(K_2)$. By the symmetries of $K_1,K_2$ and
convexity of $Z_p(K_1 \times K_2)$, it follows that:
\[
Z_p(K_1 \times K_2) \subseteq Z_p(K_1) \times Z_p(K_2) ~.
\]
On the other hand, an elementary argument ensures that:
\[
Z_p(K_1 \times K_2) \supseteq conv(Z_p(K_1) \times \set{0} , \set{0}
\times Z_p(K_2)) \supseteq \frac{1}{2} \brac{ Z_p(K_1) \times
Z_p(K_2) } ~.
\]
\end{proof}


\begin{cor} \label{cor_2344}
For any $1/n \leq \lambda \leq 1/2$:
$$
\diam(Z_{\lambda n}(T_{\lambda})) \leq C \sqrt{\lambda n} ~.
$$
\end{cor}
\begin{proof}
By Lemma \ref{lem_2335} we see that:
$$
\diam(Z_{\lambda n}(T_{\lambda})) \leq \diam(Z_{\lambda n}(
K_{\lfloor \lambda n \rfloor}))  + \diam(Z_{\lambda n}(D_{\lceil
(1-\lambda) n \rceil})) ~.
$$
Observe that $\diam(Z_{\lambda n}(K_{\lfloor \lambda n \rfloor})) \leq \diam(K_{\lfloor \lambda n \rfloor}) \leq 20 \sqrt{\lambda n}$. As for the other
summand, a straightforward computation reveals that when $1/n \leq \lambda \leq 1/2$:
$$
Z_{\lambda n}(D_{\lceil (1-\lambda) n \rceil}) \simeq \sqrt{\lambda} \sqrt{n} B_{\lceil (1-\lambda) n \rceil} ~.
$$
The assertion now follows.
\end{proof}

Recall that for any isotropic convex body $K \subset \RR^n$:
\begin{equation}
 q^{\sharp}(K) = q^\sharp(\mu_K) := \sup \set{q \geq 1 ; \diam(Z_q(K)) \leq c^\sharp \sqrt{n}} ~, \label{eq_2351}
\end{equation}
where $c^\sharp > 0$ is an appropriate universal constant (as in
Section \ref{sec4}).

\begin{cor} \label{cor_2355}
For any $n \geq 1$, there exists a centrally-symmetric isotropic convex body $K \subset \RR^n$, such that:
\begin{enumerate}
\item[(a)] $\displaystyle q^{\sharp}(K) \geq c n$; and
\item[(b)] $\displaystyle \log L_K \geq c \log L_n$,
\end{enumerate}
where $c > 0$ is a universal constant.
\end{cor}

\begin{proof}
Take $\lambda_0 := \min \{ (c^\sharp / C)^2,1/2 \}$, where $C$ is the
constant from Corollary \ref{cor_2344}. Then $K = T_{\lambda_0}$
satisfies the first assertion in view of the choice of $\lambda_0$, and by
(\ref{eq:same-LK}):
$$
L_K \simeq L_{\lfloor \lambda_0 n \rfloor}^{\lambda_0} \gtrsim
 L_n^{\lambda_0} ~,
$$
where the inequality $L_{\lfloor \lambda n \rfloor} \gtrsim L_n$ for
any $0 < c \leq \lambda \leq 1$ follows from the techniques in
\cite[Section 3]{BKM}. Since $L_K \geq c  > 0$, the second assertion
follows.
\end{proof}

\begin{proof}[Proof of Theorem \ref{thm:SlicingEquiv}]
If $L_n \leq A$, then $\vol(K)^{\frac{1}{n}} \geq 1/A$ for any isotropic convex body $K \subset \Real^n$. Consequently, by the Lutwak--Yang--Zhang lower-bound (\ref{eq:LYZ-bound}), we even have:
\[
\vr(Z_p(K)) \geq \frac{c}{A} \sqrt{p} \quad \quad \forall 1 \leq p \leq n ~.
\]

For the other direction, apply our assumption
(\ref{eq:SlicingEquivAssumption}) to the isotropic convex body $K
\subset \Real^n$ from Corollary \ref{cor_2355}, and obtain:
\[
 \frac{\sqrt{p}}{B} \leq \vr(Z_p(K)) \leq \vr(K) \simeq \frac{\sqrt{n}}{L_K} \quad \quad \forall 1 \leq p \leq q^\sharp(K) / B ~.
\]
Corollary \ref{cor_2355} then implies that:
\[
L_n \leq (L_K)^C \leq \brac{C' B^{\frac{3}{2}} \sqrt \frac{n}{q^\sharp(K)}}^C \leq C_1 B^{C_2} ~,
\]
as required.
\end{proof}




\begin{thebibliography}{99}
\def\cprime{$'$}

\bibitem{BallPhD}
K.~Ball.
\newblock PhD thesis, Cambridge, 1986.

\bibitem{Ball-kdim-sections}
K.~Ball.
\newblock Logarithmically concave functions and sections of convex sets in
  $\mathbb{R}^n$.
\newblock {\em Studia Math.}, 88(1):69--84, 1988.


\bibitem{ber} L. Berwald.
\newblock{Verallgemeinerung eines Mittelwertsatzes von J. Favard,
 f\"ur positive konkave Funktionen}.
 \newblock{\em Acta Math.}, 79:17-–37, 1947.

\bibitem{bobkov} S. G. Bobkov. \newblock{On concentration of distributions of random weighted
sums.} \newblock{\em Ann. Prob.}, 31(1):195--215, 2003.

\bibitem{Borell-logconcave}
Ch. Borell.
\newblock Convex measures on locally convex spaces.
\newblock {\em Ark. Mat.}, 12:239--252, 1974.


\bibitem{bou_amer} J. Bourgain. \newblock{On high-dimensional maximal
functions associated to convex bodies.}  \newblock{\em Amer. J.
Math.}, 108(6):1467--1476, 1986.

\bibitem{bou_cong} J. Bourgain. \newblock{Geometry of Banach spaces and harmonic analysis.}
\newblock{ \em Proceedings of the International Congress of Mathematicians,
(Berkeley, Calif., 1986), Amer. Math. Soc., Providence, RI, }
871--878, 1987.


\bibitem{bou_L} J. Bourgain. \newblock{ On the distribution of polynomials on
high dimensional convex sets.}
\newblock In {\em Geometric Aspects of Functional Analysis (Israel Seminar, 1989--90)}, volume 1469 of
  {\em Lecture Notes in Mathematics}, pages 127--137. Springer, 1991.




\bibitem{bou_psi2} J. Bourgain.  \newblock{ On the isotropy-constant problem
for ``Psi-$2$'' bodies.}
\newblock In {\em Geometric Aspects of Functional Analysis (Israel Seminar, 2001--02)}, volume 1807 of
  {\em Lecture Notes in Mathematics}, pages 114--121. Springer, 2002.



\bibitem{BKM}
J.~Bourgain, B.~Klartag and V.~Milman.
\newblock Symmetrization and isotropic constants of convex bodies.
\newblock In {\em Geometric Aspects of Functional Analysis (Israel Seminar, 2002--03)}, volume 1850 of
  {\em Lecture Notes in Mathematics}, pages 101--116. Springer, 2004.


\bibitem{DP} N.~Dafnis and G.~Paouris. \newblock{Small ball probability estimates, $\psi_2$-behavior and the hyperplane conjecture.}
\newblock{\em J. of Funct. Anal.}, 258:1933--1964, 2010.

\bibitem{fradelizi} M.~Fradelizi. \newblock{Sections of convex bodies through their
centroid.} \newblock{\em Arch. Math.}, 69(6):515--522, 1997.

\bibitem{LYZ} E. Lutwak, D. Yang and G. Zhang. \newblock{$L^p$ affine isoperimetric inequalities.} \newblock {\em J. Diff. Geom.},  56:111–-132, 2000.


\bibitem{K_jfa} B.~Klartag. \newblock{An isomorphic version of the slicing problem.} \newblock{ \em J. Funct. Anal.},
218:372--394, 2005 .


\bibitem{quarter}
B.~Klartag.
\newblock On convex perturbations with a bounded isotropic constant.
\newblock {\em Geom. and Funct. Anal.}, 16(6):1274--1290, 2006.


\bibitem{K_psi2} B. Klartag. \newblock{Uniform almost sub-gaussian estimates for linear functionals
on convex sets.} \newblock{\em Algebra i Analiz (St. Petersburg
Math. Journal)}, 19(1):109--148, 2007.


\bibitem{LW} R. Lata\l a and J. O. Wojtaszczyk. \newblock{
On the infimum convolution inequality. } \newblock{\em Studia Math.}
189(2):147–-187, 2008.

\bibitem{LMS}
A.~E. Litvak, V.~Milman and G.~Schechtman.
\newblock Averages of norms and quasi-norms.
\newblock {\em Mathematische Annalen}, 312:95--124, 1998.

\bibitem{LutwakZhang-IntroduceLqCentroidBodies}
E.~Lutwak and G.~Zhang.
\newblock Blaschke-{S}antal\'o inequalities.
\newblock {\em J. Differential Geom.}, 47(1):1--16, 1997.

\bibitem{Mil71} V. Milman. \newblock{A new proof of A. Dvoretzky's theorem on cross-sections of convex bodies. (Russian)}
\newblock{\em Funkcional. Anal. i Prilo\v zen.},  5(4):28--37, 1971.
 English transl. \newblock{ \em Functional Anal. Appl.} 5:288--295, 1971.

\bibitem{MilmanPajor}
V.~D. Milman and A.~Pajor.
\newblock Isotropic position and interia ellipsoids and zonoids of the unit
  ball of a normed $n$-dimensional space.
\newblock In {\em Geometric Aspects of Functional Analysis (Israel Seminar, 1987--88)}, volume 1376 of
  {\em Lecture Notes in Mathematics}, pages 64--104. Springer, 1991.


\bibitem{Milman-Schechtman-Book}
V.~D. Milman and G.~Schechtman.
\newblock {\em Asymptotic theory of finite-dimensional normed spaces}, volume
  1200 of {\em Lecture Notes in Mathematics}.
\newblock Springer-Verlag, Berlin, 1986.
\newblock With an appendix by M. Gromov.

\bibitem{MilmanSchechtmanSharpDvorDim} V.D. Milman and G.  Schechtman. \newblock{
Global versus local asymptotic theories of finite-dimensional normed spaces.}
\newblock{\em Duke Math. J.} 90(1):73-–93,  1997.

\bibitem{Paouris-Small-Diameter}
G.~Paouris.
\newblock $\psi_2$-estimates for linear functionals on zonoids.
\newblock In {\em Geometric Aspects of Functional Analysis (Israel Seminar, 2001--02)}, volume 1807 of
  {\em Lecture Notes in Mathematics}, pages 211--222. Springer, 2003.


\bibitem{Paouris-IsotropicTail}
G.~Paouris.
\newblock Concentration of mass on convex bodies.
\newblock {\em Geom. Funct. Anal.}, 16(5):1021--1049, 2006.

\bibitem{PaourisSmallBall}
G.~Paouris.
\newblock Small ball probability estimates for log-concave measures.
\newblock To appear in {\em Trans. Amer. Math. Soc.}

\bibitem{Pis} G. Pisier. \newblock{The Volume of Convex Bodies and Banach Space Geometry.} \newblock{ \em Cambridge
Tracts in Mathematics 94}, 1989.

\bibitem{RS} C. A. Rogers and G. C. Shephard. \newblock{ The difference body of a convex body.} \newblock{\em Arch.
Math.} 8:220-–233, 1957.


\bibitem{Schneider-Book}
R.~Schneider.
\newblock {Convex bodies: the {B}runn-{M}inkowski theory}, volume~44 of
  {\em Encyclopedia of Mathematics and its Applications}.
\newblock Cambridge University Press, Cambridge, 1993.


\end{thebibliography}

\end{document}